\documentclass[11pt]{amsart}
\usepackage[utf8]{inputenc}
\usepackage[margin=1in]{geometry}
\usepackage{graphicx}
\usepackage{amsmath,amsfonts,amssymb,latexsym,amsthm}
\usepackage{epstopdf}
\usepackage{enumerate}
\usepackage[title]{appendix}

\usepackage[dvipsnames]{xcolor}

\usepackage[colorlinks=true, pdfstartview=FitV, linkcolor=RoyalBlue,citecolor=ForestGreen, urlcolor=blue]{hyperref}

\definecolor{labelkey}{rgb}{0,0,1}

\newtheorem{theorem}{Theorem}
\newtheorem{corollary}[theorem]{Corollary}
\numberwithin{equation}{section}

\newtheorem{lemma}[theorem]{Lemma}
\newtheorem{proposition}[theorem]{Proposition}

\newtheorem{remark}[theorem]{Remark}

\usepackage{amsthm}
\usepackage{chngcntr}
\usepackage{lipsum}  

\begin{document}
   \title{Regularity of solution maps of the generalized surface quasi-geostrophic equations}
	\author[]{ Gerard Misio\l{}ek}
	\address[GM]{Department of Mathematics, 
    University of Notre Dame, Notre Dame, IN 46556, USA 
    and 
    Institute for Mathematical Sciences, 
    Stony Brook University, Stony Brook, NY 11794-3660, USA 
	}
	\email{gmisiole@nd.edu }
	\author[]{ Xuan-Truong Vu}
	\address[TV]{Department of Mathematics, Statistics, and Computer Science, University of Illinois at Chicago, Chicago, Illinois 60607, U.S.A.
	}
	\email{tvu25@uic.edu}
  	\author[]{Tsuyoshi Yoneda}
	\address[TY]{Graduate School of Economics, Hitotsubashi University, 2-1 Naka, Kunitachi, Tokyo 186-8601, Japan
	}
	\email{t.yoneda@r.hit-u.ac.jp}

	\begin{abstract}
We study regularity properties of the data-to-solution maps of the family of generalized surface quasi-geostrophic equations which includes both the 2D incompressible Euler and the standard surface quasi-geostrophic equations. 
We prove that the Lagrangian solution maps, interpreted as Riemannian exponential maps on the group of exact Sobolev class diffeomorphisms, are real analytic 
and, consequently, the Cauchy problems are locally well-posed in the sense of Hadamard. 
On the other hand, we also show that the corresponding Eulerian solution maps are nowhere locally uniformly continuous on bounded subsets in the Sobolev topology and fail to be continuous in the standard (large-) H\"older topologies. 
These results sharpen earlier theorems and further highlight the striking dichotomy between regularity properties of the solution maps in the Lagrangian and Eulerian formulations. 
\end{abstract}
    \maketitle
\section{Introduction}\label{intro}

Consider the Cauchy problem for the family of 
generalized SQG equations 
\begin{equation} \label{eq:gSQG} 
\begin{cases} 
\partial_t \theta+(u\cdot\nabla)\theta = 0, 
\qquad 
x\in \mathbb{R}^2,\ t\in \mathbb{R} 
\\ 
\theta(x,0)=\theta_{0} (x), 
\end{cases} 
\end{equation} 
where $\nabla^{\perp}=(-\partial_2, \partial_1)$ 
is the symplectic gradient and where the velocity field 
$u=u(t,x)$ 
and the scalar function 
$\theta=\theta(t,x)$ 
are related by 
\begin{equation*} \label{vel} 
u=-\nabla^{\perp}(-\Delta)^{-1+\frac{\beta}{2}}\theta 
\end{equation*} 
with $\beta \ge 0$. 
Equations of this type have been studied extensively in recent years, see e.g., \cite{bauer2024geometric}, \cite{castro2025unstable}, \cite{CCCGW12},
\cite{CJK25},
\cite{constantin2016contrast},  \cite{CMT94}, \cite{inci2015regularity}, \cite{jolly2021sqg}, \cite{MV2023}, \cite{Wash},  \cite{YZJ}.

Our main goal in this paper is to study regularity properties of 
the solution map of the Cauchy problem \eqref{eq:gSQG}. 
In \cite{MV2023} the authors showed that each of the generalized SQG equations has a Lagrangian reformulation as a system of ODE 
on the group of exact diffeomorphisms preserving 
the volume form $\mu$ of the underlying domain 
which can be viewed as the configuration space of 
the dynamical system defined by \eqref{eq:gSQG}. 
When completed in the Sobolev $H^s$ norm with $s>2$ 
this space becomes a smooth Banach manifold 
$\mathcal{D}_{ex}^s$ whose tangent space 
$T_{e}\mathcal{D}_{ex}^s$ at the identity map $e(x)=x$ 
consists of divergence free vector fields characterized by 
single-valued stream functions. 
Moreover, it is also a topological group 
under composition of diffeomorphisms. 
The data-to-solution map associated with the Lagrangian reformulation of \eqref{eq:gSQG} turns out to be the (infinite-dimensional) Riemannian exponential map of the right-invariant metric defined at the identity $e$ by the inner product
\begin{align} \label{eq:beta-met} 
\langle v, w \rangle_{\dot{H}^{ -\beta/2 } } 
&= 
\int 
\phi_v \, (-\Delta)^{\frac{2-\beta}{2}} \phi_w 
\; d\mu, 
\quad 
v, w\in T_{e}\mathcal{D}_{ex}^s 
\end{align} 
where $v = \nabla^\perp\phi_v$, 
$w = \nabla^\perp\phi_w$ 
and $\phi_v, \, \phi_w$ 
are the corresponding stream functions 
of class $H^{s+1}$. 
More precisely, we have 
\begin{equation} \label{eq:exp} 
 u_0 \mapsto \exp_e^{_{\beta}}t u_0 := \gamma(t) 
\end{equation} 
where $t\mapsto\gamma(t)$ is the unique geodesic curve 
of diffeomorphisms 
starting from the identity $\gamma(0)=e$ 
in the direction of $\dot{\gamma}(0)= u_0$. 
Our first result is 
\begin{theorem} \label{analytic} 
For $s > 2$ and any $0\le \beta\le 1$ 
the (Lagrangian) solution map 
$u_0\mapsto \exp_e^{_{\beta}} u_0$ 
from $ T_{e}\mathcal{D}_{ex}^s$ 
to $\mathcal D_{ex}^s$ is analytic. 
\end{theorem} 
with its immediate consequences 
\begin{corollary} \label{cor:locDiff} 
For each $s>2$ and $0\le\beta\le 1$ 
the Riemannian exponential map is a local diffemorphism 
from a neighborhood of $0$ in $T_{e}\mathcal{D}_{ex}^s$ 
to a neighborhood of $e$ in $\mathcal{D}_{ex}^s$. 
\end{corollary} 
and
\begin{corollary} \label{lwp} 
For each $s > 2$ and $0\le\beta\le 1$ 
the generalized SQG equation in \eqref{eq:gSQG} 
is locally well-posed in $H^s$ in the sense of Hadamard. 
\end{corollary}
In \cite{BL} Bourgain and Li outlined a strategy for proving nowhere uniform continuity of the solution map for equations which are invariant under Galilean transformations. They also suggested that a functional analytic setup combining both Eulerian and Lagrangian frameworks may be useful, when applicable, 
to study equations whose velocity drifts are more singular than that of the 2D Euler equations.  
In \cite{MV2023}, using continuity arguments, 
the authors proved that 
the Eulerian data-to-solution map of \eqref{eq:gSQG} 
is not uniformly continuous in bounded sets in $H^s$ 
in the case of all equations 
that are small $\beta$-perturbations of the incompressible $2$D Euler equations. 
In this paper we will sharpen that result as follows. 
\begin{theorem}\label{nonuniform} 
Let $s>2$. For any $0\le \beta\le 1$ 
the corresponding (Eulerian) solution map 
$\theta_0 \mapsto \theta (t)$ 
is nowhere locally uniformly continuous 
as a map from bounded sets in $H^s$ 
to $ C\big([0,T];H^s(\mathbb  R^2)\big)$ 
where $T>0$ is a local existence time. 
\end{theorem} 
Theorem \ref{nonuniform} shows that the loss of uniform continuity of the Eulerian data-to-solution map, previously observed 
in~\cite{MV2023} for small values of~$\beta$, 
persists for the entire family. 
Together with Theorem \ref{analytic} and its corollaries, 
it further underscores the contrast between 
the Lagrangian and Eulerian formulations of 
the generalized SQG equations, namely, 
while the Lagrangian geodesic flow depends analytically on the initial data, 
the corresponding Eulerian evolution exhibits 
a sharp form of illposedness with respect to Sobolev perturbations. 

Our approach combines geometric constructions 
involving the group of exact volume-preserving diffeomorphisms and its right-invariant metric 
with perturbative arguments motivated by the approach of Bourgain and Li \cite{BL,BL2015} and 
Inci \cite{inci2015regularity}. 
It seems that the analytic dependence of 
the Riemannian exponential map in Theorem \ref{analytic} 
provides a convenient setting in which various other 
illposedness properties can be established systematically. 

Lastly, we turn to the Cauchy problem \eqref{eq:gSQG} 
in the setting of the H\"older $C^\alpha$ spaces. 
In a recent paper Choi, Jung and Kim \cite{CJK25} 
proved existence of unique local in time solutions 
with initial data in $L^1\cap C^\alpha$ and $\beta<\alpha<1$. 
The authors also derived a Lipschitz-type estimate for 
the corresponding solutions in Sobolev norms 
of negative index $\beta-1$. 
However, it turns out that the Eulerian solution map 
fails to be continuous in the $C^\alpha$ topology, 
which partially addresses the question and sharpens 
the statements in Jeong \cite{I-JJ}. 
Our result can be stated for general domains. 
\begin{theorem} \label{holder case} 
Assume that the functions $\theta$ and $\theta_n$ satisfy 
the following conditions 
\begin{equation} \label{general setting} 
\begin{cases} 
&\theta(t,x)=(D\eta)\theta_0\circ\eta^{-1}, 
\\ 
&\theta_n(t,x)=(D\eta_n)\theta_0\circ\eta^{-1}_n, 
\\ 
&\eta\in C^1([0,T]:C^{1+\alpha}) 
\quad \text{with} \quad \text{det}(D\eta)=1, 
\\ 
&\theta_0(x_0)=0 
\quad \text{and} \quad 
\theta_0(B(r,x_0))\in C^\alpha\setminus c^\alpha 
\quad \text{for any sufficiently small} \quad r>0, 
\\ 
&\eta_n=\eta-h_nt 
\quad (h_n\to 0), 
\end{cases} 
\end{equation} 
as well as the following alignment condition 
\begin{equation}\label{alignment}
\theta_0(x)\cdot\theta_0(y)>0 
\quad x,y\in B(x_0,r) \quad \text{for some small} \quad r>0.
\end{equation} 
Then there exists $\epsilon_0>0$ such that 
$$ 
\|\theta_n(t)-\theta(t)\|_\alpha>\|D\eta^{-1}\|^{-1}\epsilon_0 
$$ 
for $t\in(0,T]$,
where\footnote{In the 2D case this is always $1$.} 
\begin{equation}\label{lower bound}
\|D\eta^{-1}\|:=\inf_{t\in[0,T]}\sup_{|b|=1}|D\eta^{-1}(t)b|
\end{equation}
which is positive due to invertibility and continuity of $D\eta$.
\end{theorem} 
\noindent 
The last of the conditions in \eqref{general setting} 
is often referred to as a Galilean boost. 

\medskip

The paper is organized as follows. 
In Section 2 we recall the geometric formulation of the generalized SQG equations and derive the equivalent Lagrangian system on the diffeomorphism group. 
In Section 3 we prove analyticity 
of the exponential map and the corresponding 
local well-posedness properties 
followed by the non-uniform dependence result 
in Section 4
- all these results in the setting of Sobolev spaces. 
In Section 5 we show that the solution map fails to be continuous in the standard H\"olderian norms. 
Finally, the Appendix summarizes a few relevant facts 
concerning the analytic framework in Banach spaces and collects auxiliary inequalities involving fractional Sobolev norms used in the main arguments.

\section{Technical preliminaries and Lagrangian reformulation}	\label{geometricappro}

Our first task is to rewrite the SQG system \eqref{eq:gSQG} in the form which is convenient for our subsequent analysis. Let $S_{\beta,k}=R_k(-\Delta)^{-\frac{1-\beta}{2}}$ where $R_k=\partial_k(-\Delta)^{-1/2}$, ($k=1,2$) are the standard Riesz transforms, so that 
$u = - S_\beta^\perp\theta$ 
with components 
\begin{align}\label{eq:vel_comp}
u_1=S_{\beta,2} \theta \quad \text{ and }\quad u_2 
= 
-S_{\beta,1} \theta. 
\end{align}

From \eqref{eq:vel_comp} and the first of the equations in \eqref{eq:gSQG}, we obtain
	\begin{align}\label{gSQG1coor}
		\partial_t u_1 +u\cdot \nabla u_1
		=[ u\cdot \nabla,S_{\beta,2}]\theta
	\end{align}
and similarly 
	\begin{align}\label{gSQG2coor}
		\partial_t u_2 +u\cdot \nabla u_2
		=-[ u\cdot \nabla,S_{\beta,1}]\theta.
	\end{align}
Eliminating $\theta$ with the help of
	\begin{align}\label{eq:theta-u}
	R_1u_2-R_2u_1&=\big(R_1S_{\beta1}+R_2S_{\beta,2}\big)\theta=(-\Delta)^{-\frac{1-\beta}{2}}\theta
	\end{align} 
leads to the equation for the velocity field in Eulerian coordinates
	\begin{align}\label{rewriteeqn}
		\partial_t u +u\cdot \nabla u=\Bigg(
		\begin{matrix}
			[ u\cdot \nabla,-S_{\beta,2}]\\
			[ u\cdot \nabla,S_{\beta,1}]
		\end{matrix}\Bigg)(-\Delta)^{\frac{1-\beta}{2}}(R_2u_1-R_1u_2)=:q(u,u).
	\end{align}

		Equation \eqref{rewriteeqn} together with the initial condition 
		\begin{equation}\label{eq:IC}
		u(0)=u_0
		\end{equation} is an equivalent formulation of the Cauchy problem \eqref{eq:gSQG}.  That the latter implies the former follows directly from the above derivation. On the other hand, if $u$ is a solution of the Cauchy problem \eqref{rewriteeqn}-\eqref{eq:IC} then $\theta$ given by \eqref{eq:theta-u} satisfies the first equation in \eqref{eq:gSQG}. In fact, applying appropriate Riesz transforms to \eqref{gSQG1coor} and \eqref{gSQG2coor} and substituting into \eqref{eq:theta-u} after differentiating with respect to $t$ variable gives
        \begin{align}
         (-\Delta)^{-\frac{1-\beta}{2}}\partial_t\theta&=-\big(R_1S_{\beta,1}+R_2S_{\beta,2}\big) u\cdot \nabla\theta\\
            &=(-\Delta)^{-\frac{1-\beta}{2}}u\cdot \nabla\theta. \notag  
        \end{align}
        The equivalence of the two formulations is now a consequence of the following 

	\begin{lemma}\label{lemma_div}
		Let $s > 2$ and $T > 0$. If $u \in C([0,T];H^s)$ is a solution to \eqref{rewriteeqn} with  $\operatorname{div} u_0 = 0$, then 
		$\operatorname{div} u(t) = 0$ for all $0\le t\le T$.
	\end{lemma}
	\begin{proof}
		
We adapt here an argument from \cite{inci2015regularity}. From \eqref{rewriteeqn} we have 	
		\[
		u(t) = u_0 + \int_0^t \Big(q\big(u(s),u(s)\big) - u(s) \cdot \nabla u(s) \Big) \; ds
		\]
		 for $0 \leq t \leq T$. Let $\Phi(t)= R_1 u_1(t) +  R_2 u_2(t)$. Note that as $u_0$ is divergence free, we have $\Phi(0)=0$. From \eqref{gSQG1coor} and \eqref{gSQG2coor} we have
		\begin{align*}
			\partial_t \Phi &=  R_1\Big([ u\cdot \nabla,S_{\beta,2}](-\Delta)^{\frac{1-\beta}{2}}(R_2u_1-R_1u_2)\Big) -  R_2\Big([ u\cdot \nabla,S_{\beta,1}](-\Delta)^{\frac{1-\beta}{2}}(R_2u_1-R_1u_2)\Big)\\
			&\hspace{7.6cm} - R_1(u \cdot \nabla u_1)-R_2(u \cdot \nabla u_2).
		\end{align*}
		We consider these above terms on the right hand side separately. We have
		\begin{align*}
			R_1\Big([ u\cdot \nabla,S_{\beta,2}](-\Delta)^{\frac{1-\beta}{2}}(R_2u_1-R_1u_2)\Big)
			=& R_1 u_1  R_{2}^2  \partial_1 u_1 - R_1 S_{\beta,2} u_1 R_2 (-\Delta)^{\frac{1-\beta}{2}}\partial_1 u_1 \\
			& +R_1 u_2  R_{2}^2 \partial_2 u_1  -  R_1 S_{\beta,2} u_2 R_2 (-\Delta)^{\frac{1-\beta}{2}}\partial_2 u_1  \\
			& -R_1 u_1  R_{1} R_{2}  \partial_1 u_2 + R_1 S_{\beta,2} u_1 R_1 (-\Delta)^{\frac{1-\beta}{2}}\partial_1 u_2  \\
			& - R_1  u_2  R_{1} \mathcal R_{2} \partial_2 u_2 + R_1 S_{\beta,2} u_2 R_1 (-\Delta)^{\frac{1-\beta}{2}} \partial_2 u_2.
		\end{align*}

		Similarly we have
		\begin{align*}
			R_2\Big([ u\cdot \nabla,S_{\beta,1}](-\Delta)^{\frac{1-\beta}{2}}(R_2u_1-R_1u_2)\Big) 
			&=  R_2 u_1 R_{1} R_2 \partial_1 u_1  - S_{\beta,1} R_2 u_1 R_2 (-\Delta)^{\frac{1-\beta}{2}}\partial_1 u_1 \\
			& + R_2 u_2 R_{1} R_{2}\partial_2 u_1 - S_{\beta,1} R_2 u_2 R_2  (-\Delta)^{\frac{1-\beta}{2}}\partial_2 u_1 \\
			& - R_2 u_1 R_{1}^2 \partial_1 u_2 + S_{\beta,1} R_2 u_1 R_1 (-\Delta)^{\frac{1-\beta}{2}}\partial_1 u_2 \\
			& - R_2 u_2 R_{1}^2\partial_2 u_2 + S_{\beta,1}R_2 u_2 R_1 (-\Delta)^{\frac{1-\beta}{2}}\partial_2 u_2
		\end{align*}
		and
		\begin{align*}
			-R_1(u \cdot \nabla u_1)-R_2(u \cdot \nabla u_2) =	-R_1 u_1( -R_1^2-R_2^2) \partial_1 u_1 - R_1 u_2 (-R_1^2-R_2^2) \partial_2 u_1 \\
			-R_2 u_1(-R_1^2-R_2^2) \partial_1 u_2 - R_2 u_2 (-R_1^2-R_2^2) \partial_2 u_2, 
		\end{align*}
			
		
		since $R_1^2+R_2^2=-id$.

	Combining all these identities and rearranging the terms gives
		\begin{align}
			\partial_t \Phi 
			&=R_1( u_1 R_1^2 \partial_1 u_1) + R_1 (u_2 R_1^2 \partial_2 u_1) + R_1 (u_1 R_1 R_2 \partial_1 u_2 )+ R_1 (u_2 R_1 R_2 \partial_2 u_2)\\
			&+R_2 (u_1 R_1 R_2 \partial_1 u_1) + R_2 (u_2 R_1 R_2 \partial_2 u_1) + R_2 (u_1 R_2^2 \partial_1 u_2) + R_2 (u_2 R_2^2 \partial_2 u_2 )\nonumber \\
			&=R_1(u_1\cdot R_1\partial_1\Phi)+R_1(u_2\cdot R_1\partial_2\Phi)+R_2(u_1\cdot R_2\partial_1\Phi)+R_2(u_2\cdot R_2\partial_2\Phi).\nonumber
		\end{align}
	Therefore, multiplying by $\Phi$ and integrating by parts, yields

		\begin{align*}
		\frac{1}{2} \partial_t \|\Phi\|_{L^2}^2&=\big\langle R_1(u_1\cdot R_1\partial_1\Phi)+R_1(u_2\cdot R_1\partial_2\Phi)+R_2(u_1\cdot R_2\partial_1\Phi)+R_2(u_2\cdot R_2\partial_2\Phi),\Phi \big\rangle_{L^2}\\
		&=\sum\limits_{k=1}^{2}\int\limits R_k(u_1\cdot R_k\partial_1\Phi)\cdot \Phi\, dx+\sum\limits_{k=1}^{2}\int\limits R_k(u_2\cdot R_k\partial_2\Phi)\cdot \Phi\, dx\\
		&=\frac{1}{2}\sum\limits_{k=1}^2\int\limits \partial_1u_1 (R_k\Phi)^2 \, dx+\frac{1}{2}\sum\limits_{k=1}^2\int\limits \partial_2u_2 (R_k\Phi)^2 \, dx\\
		&\le C \|Du\|_{\infty}\Big(\|R_1\Phi\|_{L^2}^2+\|R_2\Phi\|_{L^2}^2  \Big)\\
		&\le C \sup\limits_{[0,T]}\|u(t)\|_{H^s}\|\Phi(t)\|_{L^2}^2
		\end{align*}
where in the last step we used the Sobolev lemma. Since $\Phi(0)=0$, applying Gronwall's inequality completes the proof.
	\end{proof}
The above arguments can be summarized as follows.
	\begin{corollary}\label{prop_equivalence}
	Let $s > 2$, $T > 0$ and $\theta_0 \in H^s$. If $\theta \in C([0,T];H^s(\mathbb R^2))$ is a solution to \eqref{eq:gSQG} with $\theta(0)=\theta_0$ then $u=\big(R_2(-\Delta)^{-\frac{1-\beta}{2}}\theta,-R_1(-\Delta)^{-\frac{1-\beta}{2}}\theta\big)$ is a solution to \eqref{rewriteeqn} on $[0,T]$. On the other hand, if $u \in C([0,T];H^s)$ is a solution to \eqref{rewriteeqn} with $u(0)=\big(R_2(-\Delta)^{-\frac{1-\beta}{2}}\theta_{0},-R_1(-\Delta)^{-\frac{1-\beta}{2}}\theta_{0}\big)$, then $\theta=(-\Delta)^{\frac{1-\beta}{2}}(R_1 u_2-R_2 u_1)$ is a solution to \eqref{eq:gSQG} on $[0,T]$.
\end{corollary}	

Finally, we are ready to proceed with the Lagrangian formulation of the Cauchy problem \eqref{rewriteeqn}-\eqref{eq:IC}. To this end we employ the flow equation, namely
\begin{align} \label{eq:flow-eq} 
	\frac{d \gamma}{dt}(t,x) = u\big(t, \gamma(t,x)\big),\quad \gamma(0,x)=x
\end{align}
and rewrite the velocity equation \eqref{rewriteeqn} as a first order system on the tangent bundle $ T\mathcal{D}_{ex}^s$ of the group $\mathcal{D}_{ex}^s$, that is
\begin{align}\label{lagrange}
	\partial_t\left(\begin{matrix}
		\gamma\\
		\dot{\gamma}
	\end{matrix}\right)=\left(\begin{matrix}
		\dot{\gamma}\\
		q(\dot{\gamma}\circ \gamma^{-1},\dot{\gamma}\circ\gamma^{-1} )\circ \gamma
	\end{matrix}\right)=:F(\gamma, \dot{\gamma})
\end{align}
subject to the initial conditions
\begin{align}\label{eq:intial_lagr}
	\gamma(0)=e \quad \text{and} \quad 	\dot{\gamma}(0)= u_0
\end{align}
where $F:U \rightarrow H^s\times H^s$ is defined in some open neighborhood $U$ of the point $(e,0)$ in $ \mathcal{D}_{ex}^s\times H^s$. The flow $t\mapsto \gamma(t)$ of the velocity field $u$ describes of course the geodesic path in $\mathcal{D}_{ex}^s$ starting from $e$ in the direction $u_0$, see \eqref{eq:exp} above.

	\section{Local well-posedness}\label{section_lwp}
In this section we prove Theorem \ref{analytic} and Corollary \ref{lwp}. We will adapt the method used by Shnirelman \cite{Shn12}.  
For the definition and relevant constructions involving analytic maps in the infinite dimensional setting see the Appendix. 	
	\subsection{Proof of Theorem \ref{analytic}}
Given $u_0$ in $ T_e\mathcal{D}_{\mu, ex}^s$ let $\gamma(t)$ be the geodesic in the diffeomorphism group with $\gamma(0)=e$ and $\dot{\gamma}(0)=u_0$ and let $\theta_0$ be the corresponding initial condition in \eqref{eq:gSQG}. Observe that for any $t\ge 0$ we have the conservation law
 \begin{align}\label{eqn:conservation}
     \theta(t,\gamma(t,x))=\theta_0(x).
 \end{align}
From \eqref{eq:flow-eq}, using \eqref{eq:gSQG} and \eqref{eqn:conservation}, we obtain
 \begin{align}\label{eqn:flow_intermof_gamma}
    \frac{\partial \gamma(t,x)}{\partial t}=\left\{\mbox{curl}^{-1} \left(-\Delta\right)^{\frac{\beta}{2}}\big(\theta_0\circ \gamma^{-1}(t)\big)\right\}\circ \gamma(t,x),
 \end{align}
where $\mbox{curl}^{-1} =-\nabla^{\perp}(-\Delta)^{-1}$ is the inverse of the $\text{curl}$ operator on the space $H_0^s$ denote the space of Sobolev $H^s$ divergence free vector fields with mean zero.\footnote{ Note that $H_0^s$ corresponds to $T_e\mathcal{D}_{\mu, ex}^s$ but in the case of $\mathbb{R}^2$ we need appropriate vanishing conditions at infinity. }
 
We will construct a local chart near the identity and use Proposition \ref{prop: analytic_Shnirelman} and Proposition \ref{prop:AIFT_Shnirelman} to show that the exponential map is locally analytic between Banach spaces.
 \begin{lemma}
  Let $\varepsilon>0$ be sufficiently small. If $v\in H_0^{s}$ with $\|v\|_s<\varepsilon$ then there is a unique (up to an additive constant) function $\varphi_{v}\in H^{s+1}$ such that the map 
     \begin{align}
     	x\rightarrow \xi_v(x)=x+v(x)+\nabla \varphi_{v}(x)
     \end{align}
 is an exact volume-preserving $H^s$ diffeomorphism. Furthermore, the map $v\mapsto v+\nabla  \varphi_{v}$ is analytic  as a map from $H_0^s$ to $H^s$.
 \end{lemma}
\begin{proof}
Since  the map $x\mapsto x+v(x)$ does not in general preserve the volume form $\mu$, we need to adjust it by adding a gradient term $\nabla \varphi_{v}(x)$ which can be done with the help of the Hodge decomposition.
    Call the adjusted map $\xi_v=e+v+\nabla \varphi_{v}$ satisfying $\det (J (\xi_v))=1$. So $\xi_v$ is an element of the subgroup $\mathcal{D}_{\mu, ex}^s$ near the identity and we have
     \begin{align}\label{eqn:det_der_gamma_v}
         \det\begin{pmatrix}
             1+\dfrac{\partial v_1}{\partial x_1}+ \dfrac{\partial^2 \varphi_{v}}{\partial x_1^2} &  \dfrac{\partial v_1}{\partial x_2}+ \dfrac{\partial^2 \varphi_{v}}{\partial x_1 \partial x_2}\\
             \dfrac{\partial v_2}{\partial x_1}+ \dfrac{\partial^2 \varphi_{v}}{\partial x_2 \partial x_1} &  1+\dfrac{\partial v_2}{\partial x_2}+ \dfrac{\partial^2 \varphi_{v}}{\partial x_2^2}
         \end{pmatrix}=1.
     \end{align}

      Expanding the \eqref{eqn:det_der_gamma_v} and rewriting it we obtain an equation for the Laplacian of $\varphi_v$, that is
      \begin{align}\label{eqn:contraction_map}
      	\Delta\varphi_{v}= P(D v,D^2\varphi_{v}).
      \end{align}
  where $P$ is a polynomial of degree $2$ in $Dv$ and $D^2\varphi$, namely 
 \begin{align}
    P\big(D v,D^2\varphi\big)= -\bigg(\frac{\partial v_1}{\partial x_1}+ \frac{\partial v_2}{\partial x_2}
     &+ \frac{\partial v_1}{\partial x_1} \frac{\partial v_2}{\partial x_2}+ \frac{\partial v_1}{\partial x_1} \frac{\partial^2 \varphi}{\partial x_2^2}+ \frac{\partial v_2}{\partial x_2} \frac{\partial^2 \varphi}{\partial x_1^2} + \frac{\partial^2\varphi}{\partial x_1^2}\frac{\partial^2\varphi}{\partial x_2^2} \bigg)\notag\\
     &+  \frac{\partial v_2}{\partial x_1} \frac{\partial v_1}{\partial x_2}  + \frac{\partial v_2}{\partial x_1}  \frac{\partial^2 \varphi}{\partial x_1 \partial x_2}+ \frac{\partial v_1}{\partial x_2} \frac{\partial^2 \varphi}{\partial x_2 \partial x_1} + \bigg(\frac{\partial^2 \varphi}{\partial x_1 \partial x_2}\bigg)^2.
 \end{align}
 Let $Q(v,\varphi_{v})=  	\Delta^{-1} P(D v,D^2\varphi_{v})$.
Observe that $Q$ is analytic in both variables as a map from $H_0^s\times H^{s+1}$ to $H^{s+1}$ and, moreover, $Q$ and its G\^ateaux derivatives satisfy the inequalities
\begin{align}\label{eqn:bound_Q}
    &\|Q(v,\varphi)\|_{s+1}\le C\big(\|v\|_s^2+\|\varphi\|_{s+1}^2\big) 
\end{align}
and
\begin{align}\label{eqn:bound_nabla_Q}
    &\|\partial_vQ(v,\varphi)\|_{s+1}+\|\partial_{\varphi} Q(v,\varphi)\|_{s+1}\le C(\|v\|_s+\|\varphi\|_{s+1}),
\end{align}
with some constant $C>0$. 

Let $\varepsilon<(2C)^{-1}$. For any $v$ in $H_0^s$ with $\|v\|_{s}<\varepsilon$, if $\|\varphi\|_{s+1}\le \varepsilon$  then from \eqref{eqn:bound_Q} we have 
 \begin{align}
    	\|Q(v,\varphi)\|_{s+1}\le C(\varepsilon^2+\varepsilon^2)=2C\varepsilon^2< \varepsilon
    \end{align}
which implies that $\varphi \rightarrow Q(v,\varphi)$ maps the $\varepsilon$-ball $B_{\varepsilon}^{s+1}$ centered at the origin in $H^{s+1}$ to itself. 
   
Furthermore, from \eqref{eqn:bound_nabla_Q}, for any $\varphi_1,\varphi_2$ in $B_{\varepsilon}^{s+1}$ we have 
    \begin{align}
    \|Q(v,\varphi_1)-Q(v,\varphi_2)\|_{s+1}&\le \max_{0\le \lambda\le 1}\|\nabla Q\big(v,\lambda\varphi_1+(1-\lambda)\varphi_2\big)\| \|\varphi_1-\varphi_2\|_{s+1}\notag\\
    &\le 2C\varepsilon \|\varphi_1-\varphi_2\| _{s+1}\\
    &<  \|\varphi_1-\varphi_2\| _{s+1}.\notag
\end{align}
    Thefore, there is a ball of sufficiently small radius such that for any $v$ belongs to this ball, there exists a unique solution $\varphi=\varphi_v$ of \eqref{eqn:contraction_map} by the contraction mapping argument. 
    
    Since $Q$ is analytic, by the Analytic Implicit Function Theorem \ref{prop:AIFT_Shnirelman}, there exists a unique analytic solution $\varphi_v$ which shows that the map $v\mapsto\xi_v-e$ is analytic  as a map from $H_0^s$ to $H^s$.

 \end{proof}
This lemma defines a local analytic chart at the identity $e$ in $\mathcal{D}_{\mu,ex}^{s}$, namely
$$U\ni v\mapsto \xi_v=e+v+\nabla\varphi_v\in\mathcal{U}$$
where $U$ and $\mathcal{U}$ are open neighborhood of $0$ in $T_e\mathcal{D}_{\mu,ex}^s\simeq H_0^s$ and of $e$ in $\mathcal{D}_{\mu,ex}^{s}$.
 
Next, we rewrite the equation \eqref{eqn:flow_intermof_gamma}.
	First, consider the space 
    $$H_{0,\gamma}^s:=\Big\{ U=u\circ \gamma: u\in H^s,\; \text{div}\,u=0\;\; \text{and}\; \int_{M}u\,d\mu=0\Big\},$$
    which is the right translation of $H_0^s$ by $\gamma\in \mathcal{D}_{\mu,ex}^s.$
To see that $H_{0,\gamma}^s$ depends analytically on $\gamma$ we proceed as follows. First, observe that by the change of variables formula and the fact that $\gamma$ preserves the volume form $\mu$ we have
\begin{align}
   0= \int_{M}u\,d\mu=\int_{\gamma(M)}\gamma^\ast(U\circ\gamma^{-1}\,d\mu)=\int_{M}U\,d\mu.\notag
\end{align}
Next, we have
	\begin{align}\label{eq:divUcircgamma}
		0=\text{div}\,u&=\text{div} \,(U\circ \gamma^{-1})=\sum_{k} \partial_{k} (U^k\circ\gamma^{-1})\notag\\
        &=\sum_{k,l}^{}\partial_lU^{k}\circ \gamma^{-1}\cdot\partial_k(\gamma^{-1})^{l}
        =:L_\gamma(U).
	\end{align}	
   Observe that $L_\gamma$ defined by the right hand side of \eqref{eq:divUcircgamma} is a linear first order differential operator in $U$. Since $\gamma^{-1}\circ\gamma=e$ we have $D\gamma^{-1}\circ\gamma \cdot D\gamma=Id$ and since $\det (D\gamma)=1$ the coefficients of $L_\gamma$ depend linearly on the entries of $D\gamma$. Furthermore, the composition with $\gamma^{-1}$ is an analytic map (see e.g. Whittlesey \cite{Whittlesey1965}). 
   
   We can now describe the space $H_{0,\gamma}^s$ in terms of $U$ and $\gamma$ as
       \begin{align}
            H_{0,\gamma}^s:=\Big\{ U: L_\gamma(U)=0 \text{ and }\int U \,d\mu=0 \Big\}
    \end{align}
    and $L_\gamma$ depends on $\gamma $ analytically it follows that the dependence of the subspace $H_{0,\gamma}^s$ on $\gamma$ is analytic as well.  
    
    Next, using the conservation law \eqref{eqn:conservation} and the equation expressing the drift velocity $u$ in terms of $\theta$ in \eqref{eq:gSQG} we have
    \begin{align}\label{e:Gamma_gam}
    U=u\circ \gamma=\big\{\text{curl}^{-1}\,\left(-\Delta\right)^{\frac{\beta}{2}}(\theta_0\circ \gamma^{-1})\big\}\circ\gamma=: F(\gamma,\theta_0)
    \end{align}
       and we observe that the operator $F(\gamma,\theta_0)$ defined by the right hand side of \eqref{e:Gamma_gam} depends analytically on both $\gamma$ and $\theta_0$.
       The analyticity of $F(\gamma, \theta_0)$ in $ \gamma $ follows from the inclusion $ F(\gamma, \theta_0) \in H_{0,\gamma}^s$, which ensures that the inversion mapping $\gamma \mapsto \gamma^{-1}$ is analytic. Consequently, the composition $\gamma \mapsto \theta_0 \circ \gamma^{-1}$ is also analytic. Alternatively, we can argue by calculating the G\^ateaux derivative of the map $\gamma\mapsto \theta_0\circ \gamma$ from $\mathcal{D}_{\mu,ex}^s\rightarrow H^{s-1}$
\begin{align}
    \frac{d}{dt}(\theta_0\circ (\gamma+th)^{-1})\Big|_{t=0}&=\nabla\theta_0((\gamma+th)^{-1})\cdot \frac{d}{dt} (\gamma+th)^{-1}\Big|_{t=0}\\
    &=-\nabla\theta_0(\gamma^{-1})\cdot(D\gamma^{-1}\cdot h\circ\gamma^{-1})
\end{align}
and observe that we have
\[
\left\| \nabla \theta_0 \circ \gamma^{-1} \cdot \left[ D \gamma^{-1} \cdot h\circ\gamma^{-1} \right] \right\|_{H^{s-1}}
\leq C \left\|  \theta_0 \right\|_{H^{s-1}} \left\| \gamma \right\|_{H^{s}}.
\]
       As a result, analyticity is preserved under the linear operator $ \text{curl}^{-1} \left(-\Delta\right)^{\frac{\beta}{2}}$ and the right composition with $\gamma$.

		Our task is thus equivalent to showing that the solution of the following system in $\mathcal{D}_{\mu,ex}^s\times H^{s-1+\beta}$ depends analytically on $u_0$
		\begin{align}\label{e:gammaODE}
        	\begin{cases}
		&\dfrac{\partial \gamma}{\partial t}=F(\gamma,\theta_0)\\
        & \gamma(0)=e
        \end{cases}
		\end{align}
        where $F: \mathcal{D}_{\mu,ex}^s\times H^{s-1+\beta}\rightarrow  H^{s-1+\beta}$ is analytic in both variables. Applying Proposition \ref{prop: analytic_Shnirelman} in Appendix we obtain the solution $\gamma(t)=\exp_{e}^\beta tu_0$ with $0\le t\le T$ for some $T>0$ and $u_0=-\nabla^{\perp}(-\Delta)^{-1+\frac{\beta}{2}}\theta_0$, which depends analytically on both $t$ and $\theta_0$. This completes the proof of Theorem \ref{analytic}.

	\subsection{Proof of Corollary \ref{lwp}}

We start with the following lemma

\begin{lemma}\label{lemma_lagrangian_to_eulerian}
	Let $s >2$ and $T > 0$. Assume that $\gamma$ is a solution of \eqref{lagrange} on $[0,T]$ for the initial values $\gamma(0)=e$ and $ 	\dot{\gamma}(0)= u_0\in H^s$. Then $u$ given by
	\[
	u(t):=\dot{\gamma}(t) \circ \gamma(t)^{-1}
	\] 
	is a solution to \eqref{rewriteeqn}.
\end{lemma}

\begin{proof}
	We need to prove
	\[
	u(t) = u_0 + \int_0^t q(u(s),u(s)) - (u(s) \cdot \nabla) u(s) \;ds, \quad \forall \; 0 \leq t \leq T.
	\]
	Note that by Theorem \ref{analytic} we have $\gamma \in C^\infty([0,T];\mathcal{D}_{\mu, ex}^s)$. Therefore by Sobolev's imbedding lemma and the properties of the composition $u =\dot{\gamma} \circ \gamma^{-1}  \in C^1([0,T] \times \mathbb R^2;\mathbb R^2)$ (see e.g. \cite{Inci_K_T2013}).  Thus we have pointwise
	\[
	\ddot{\gamma} = \big(u_t + (u \cdot \nabla) u \big) \circ 	\gamma
	\]
	And from \eqref{lagrange} we conclude 
	\[
	(u_t + (u \cdot \nabla) u) \circ 	\gamma = q(u,u) \circ \gamma
	\]
	or composing on the right with $\gamma^{-1}$ we have
    \[u_t + (u \cdot \nabla) u = q(u,u).\]
    Rewriting this last equation in the integral form we get
	\[
	u(t) = u_0 + \int_0^t q(u(s),u(s))-(u(s) \cdot \nabla) u(s) \;ds \qquad 
	\]
	for any $ 0 \leq t \leq T$. Since $s>2$ by the algebra property of $H^{s-1}$ and Sobolev's imbedding lemma, the result follows.
\end{proof}

\begin{lemma}\label{lemma_eulerian_to_lagrangian}
	Let $s >  2$ and $T > 0$. If $u \in C([0,T];H^s)$ is a solution to \eqref{rewriteeqn} then its flow $t\mapsto\gamma(t)$ is a solution to \eqref{lagrange}.
\end{lemma}

\begin{proof}
	We know that given $u$ there is a unique $	\gamma \in C^1([0,T];\mathcal{D}_{\mu, ex}^s)$ with
	\begin{align}\label{e:flow}
	\dot{\gamma} = u \circ 	\gamma\quad 	\text{with}\quad\gamma(0)=e.
	\end{align}
	From the integral relation $u(t)=u_0 + \int_0^t \big(q(u,u)-(u \cdot \nabla) u\big) \;ds$ we see that $u \in C^1([0,T] \times \mathbb R^2;\mathbb R^2)$. Taking the derivative in $t$ of the differential equation in \eqref{e:flow} we get pointwise
	\[
	\ddot{\gamma} = \big(u_t + (u \cdot \nabla)u\big) \circ 	\gamma = q(u,u) \circ 	\gamma
	\]
	and consequently
	\[
	\gamma(t) = 	\dot{\gamma}(0) + \int_0^t q(	\dot{\gamma} \circ 	\gamma^{-1},	\dot{\gamma} \circ 	\gamma^{-1}) \circ \gamma \; ds \quad \forall \; 0 \leq t \leq T.
	\]
	 Therefore $	\gamma \in C^1([0,T];\mathcal{D}_{\mu,ex}^s)$ and $\gamma$ is a solution to \eqref{lagrange}.
\end{proof}

Recall from previous sections that solutions of \eqref{lagrange} with initial conditions $\gamma(0)=e$ and $\dot{\gamma}(0)= u_0 \in H^s$ can be described by the Riemannian exponential map on $\mathcal{D}_{\mu,ex}^s$ (see \eqref{eq:beta-met} and \eqref{eq:exp}). From Theorem \ref{analytic} we know that $\exp_e^{\beta}$ is real analytic. 


	\begin{proof}[Proof of Corollary \ref{lwp}]
		Take $\theta_0 \in H^s(\mathbb R^2)$ and for $u_0=(S_{\beta,2}\theta_{0},-S_{\beta,1}\theta_{0})$ let
		\[
		\gamma(t)=\exp_e^{\beta}(t u_0) \quad \mbox{and} \quad u(t)=	\dot{\gamma}(t) \circ 	\gamma(t)^{-1}
		\]
		for $0\le t \le T$ where $T>0$ comes from the Fundamental theorem of ODE as in the previous proof. Since $\gamma(t)\in \mathcal{D}_{\mu,ex}^s$ where $s>2$, we know that $\gamma(t)$ is a curve of $C^1$ diffeomorphisms by the Sobolev embedding theorem. Therefore, by properties of the composition map -- see e.g. \cite{EbMa}, \cite{Inci_K_T2013} -- this implies that $u \in C([0,T];H^s)$ for some $T > 0$. 
        
        Next, set
		\[
		\theta(t) = (-\Delta)^{\frac{1-\beta}{2}}\Big( R_1 u_2(t)-R_2 u_1(t) \Big)
		\]
		and observe that $\theta$ solves \eqref{eq:gSQG} by Lemma \ref{lemma_lagrangian_to_eulerian} and Corollary \ref{prop_equivalence}. By Theorem \ref{analytic}, $\exp_{e}^{\beta}tu_0$ is analytic in $u_0$, hence (locally) Lipschitz continuous.
        In addition, the map $\theta_0\mapsto u_0=-\nabla^{\perp}(-\Delta)^{-1+\frac{\beta}{2}}\theta_0$ is linear and bounded in $H^s$.
        Combining with the continuity of the composition map $\theta_0\mapsto \theta_0\circ \gamma^{-1}$, we obtain that the dependence on the initial condition $\theta_0$ is continuous.
        
        Uniqueness of solutions follows directly from the Fundamental ODE theorem. Indeed, suppose two solutions $\theta_1(t)$ and $\theta_2(t)$ exists for the same initial data $\theta_0$. The associated Lagrangian flows $\gamma_1(t)$ and $\gamma_2(t)$ must satisfy $\gamma_1(t)=\gamma_2(t)$ for all time as long as they defined due to uniqueness of geodesics in Theorem \ref{analytic}. Hence, $\theta_1(t)=\theta_0(\gamma_1^{-1}(t))=\theta_0(\gamma_2^{-1}(t))=\theta_2(t).$ 
	\end{proof}

\begin{remark}
    Alternatively, we can prove Corollary \ref{lwp} by directly showing that the right hand side $F(\gamma,\theta_0)$ is locally Lipschitz with respect to both variables $\gamma$ and $\theta_0$. This can be done by showing that $F(\gamma,\theta_0)$ has a bounded G\^ateaux derivative in the open neighborhood $\mathcal{U}$ of the identity $e$ in $\mathcal{D}_{\mu,ex}^s$ and then apply the mean value estimate to deduce the desired property of $F(\gamma,\theta_0)$.

\end{remark}

\begin{proof}[Proof of Corollary \ref{cor:locDiff}]
As before let $\gamma(t)=\exp_e^{\beta}tu_0$ with $u_0\in T_eD_{\mu,ex}^s$. A standard calculation (as in finite-dimensional Riemannian geometry) gives 
\begin{align}\label{e:dexp}
    d\exp_e^{\beta}(0)u_0=\dot\gamma(0)=u_0
\end{align}
which combined with the Inverse Function Theorem yields the result.
\end{proof}

\section{Non-uniform dependence of the solution map for $0\le \beta\le 1$.}\label{section_nonuniform}
In the following we fix the local existence time (see e.g., Corollary \ref{lwp}) and let 
$\mathcal{O} \subset H^s(\mathbb{R}^2)$ 
denote the set of those initial conditions $\theta_0$ for which the solution of \eqref{eq:gSQG} exists at least up to time $T=1$. 
To prove Theorem \ref{nonuniform} we use 
the following lemma. 

%
\begin{lemma} \label{lemma_dense} 
There is a dense subset 
$\mathcal{S} \subseteq \mathcal{O}$ 
consisting of functions with compact support 
such that for each function 
$\theta_0 \in \mathcal{S}$ 
we can find $x_{*} \in \mathbb R^2$ 
and $\theta_{*} \in H^s(\mathbb R^2)$ 
satisfying 
\begin{align}\label{e:2cond} 
B(x_*,2)\cap \mathrm{supp}\, \theta_0 = \emptyset	
\quad \text{and} \quad 	
\big( d\exp_{e}^{_{\beta}}(u_0)w_{*}\big)(x_*)\neq 0
\end{align}
where 
$$
u_0 
= 
\Big(S_{\beta,2}\theta_0,-S_{\beta,1}\theta_0\Big) 
\quad \text{and}\quad 
w_{*}=\Big(S_{\beta,2}\theta_*,-S_{\beta,1}\theta_*\Big) 
$$ 
and $exp_{e}^{_{\beta}}$ is the Riemannian exponential map in \eqref{eq:exp}. 
\end{lemma}
\begin{proof} 
We can assume that the initial data $\theta_0$ 
has compact support since such functions are dense 
in $H^s(\mathbb{R}^2)$. 
Let $x_\ast \in \mathbb{R}^2$ be any point whose distance 
from $\mathrm{supp}\, \theta_0$ is at least $2$. 
Expressing $S_{\beta,k}$ as a principal value integral 
\begin{align*}
S_{\beta,k} \theta(x) 
= 
\frac{1}{\sqrt{2\pi}} 
\, p.v. \int_{\mathbb R^2} 
\frac{x_k - y_k}{|x-y|^{2-\beta}}\,\theta(y) \, dy 
\qquad (k=1,2) 
\end{align*} 
we see that it is always possible to choose 
a smooth positive bump function $\theta_\ast$ 
supported near $x_\ast$ such that 
$w_\ast(x_*) 
=  
\big( 
S_{\beta,2}\theta_\ast (x_\ast),-S_{\beta,1}\theta_\ast (x_\ast) 
\big) 
\ne 0$.

Next, consider the derivative of the exponential map 
evaluated at $x_\ast$ and $w_\ast$ above 
\begin{align}\label{e:diff_exp}
t \mapsto \epsilon (t) 
= 
\big( d \exp_{e}^{_{\beta}} (tu_0)w_\ast \big)(x_*)
\end{align}
as an analytic function of the time variable. 
From \eqref{e:dexp} 
(cf. the  proof of Cor. \ref{cor:locDiff}) 
we have that $d \exp_{e}^{_{\beta}} (0)= \mathrm{id}$ 
which implies that $\epsilon (t)$ is not identically zero 
and, consequently, we can choose a sequence of points 
in the interval $(0,1)$ with $t_n \nearrow 1$ 
such that 
$\epsilon(t_n) \neq 0$ for $n=1, 2 \dots$. 
Letting now $\mathcal{S}$ be the set which consists of 
the functions 
of the form $t_n \theta_0$ where $\theta_0$ has compact support we see that both conditions in \eqref{e:2cond} are trivially satisfied and the lemma follows.    
\end{proof}

The following inequalities apply to functions whose compact supports do not overlap. For any $s\ge 0$ there exists a constant $C > 0$ such that given any $ x,y $ in $\mathbb R^2$ satisfying $0<r={|x-y|}/{4} <1$, the inequality
\begin{equation}\label{ineq2}
	\|f_1+f_2\|_{H^s} \geq C(\|f_1\|_{H^s}+\|f_2\|_{H^s})
\end{equation}
holds for all functions $f_1,f_2 \in H^{s}(\mathbb R^2)$ where $f_1$ is supported in the ball $B(x,r)$ centered at $x$ and $f_2$ is supported in the ball $B(y,r)$ centered at $y$ (See Lemma \ref{lemma_vanishing_support1} in the Appendix for the proof.).\\

We next proceed with the proof of 
the main theorem.
\begin{proof}[Proof of Theorem \ref{nonuniform}]
Without loss of generality we assume that $T=1$. Given any $\theta_0$ in $\mathcal{O}$ we aim to show that there exists $R_* > 0$ such that for all $R$ in the interval $(0, R_*)$ the data-to-solution map of \eqref{eq:gSQG} is not uniformly continuous in 
$\mathcal{O}\cap B^{s}(\theta_0,R)$, 
where $B^s(\theta_0,R)$ stands for the open ball in $H^s$ centered at $\theta_0$ of radius $R$. 
To this end we will construct two sequences 
of initial data 
$\theta_0^{(1,n)}$ and $\theta_0^{(2,n)}$ 
where $n = 1, 2 \dots$ 
confined in $B^s(\theta_0,R)$ 
and satisfying 
\begin{align} 
\lim\limits_{n \to \infty} \big\|\theta_0^{(1,n)} - \theta_0^{(2,n)}\big\|_{H^s} 
= 0 
\end{align} 
and 
\begin{align} 
\limsup_n \big\| 
\theta_0^{(1,n)} \circ (\gamma^{(1,n)})^{-1} 
- 
\theta_0^{(2,n)} \circ(\gamma^{(2,n)})^{-1} 
\big\|_{H^s} > 0 
\end{align} 
where 
$t \to \gamma^{(i,n)}$ 
are the flow maps of the velocity fields corresponding to 
$$ 
u_0^{(i,n)} 
= 
-\nabla^\perp(-\Delta)^{-1+\frac{\beta}{2}} \theta_0^{(i,n)}, 
\qquad 
\text{where} 
\;\; n = 1, 2 \dots 
\;\; \text{and} \;\; i=1, 2. 
$$ 
In what follows it will be convenient to introduce 
the notation 
$E_e^{\beta}(\theta)=\exp_{e}^{\beta} u$. 
From Lemma \ref{lemma_dense} we can find 
a nonzero $H^s$ vector field $w_*$ 
and a point $x_*$ in $\mathbb R^2$ located far away 
from the support of $\theta_0$ such that 
\begin{align}\label{e:fromlocdiff}
\big| 
\big( dE_{e}^{_{\beta}}({\theta_0})(w_*)\big)(x_*) 
\big| 
\geq 
\kappa_* \|w_*\|_{H^s} 
\end{align} 
for some constant $\kappa_* > 0$. 

Furthermore, recall that by continuity of the composition operator and the Sobolev lemma there is a constant $C>0$ such that
	\begin{align}
		&\frac{1}{C} \|f\|_{H^s} \leq \|f \circ \gamma^{-1}\|_{H^s} \leq C \|f\|_{H^s} \label{bdd}
    \end{align}
	for any $f \in H^s(\mathbb R^2)$ and any transformation $\gamma$ in $E_{e}^{_{\beta}}(B_{}^s(\theta_0,R_1))$ and where $R_1 > 0$ is chosen so that $B_{}^s(\theta_0,R_1)\subset \mathcal{O}$.
    
Let $\gamma_0= E_{e}^{_{\beta}}(\theta_0)$ and observe that it follows from \eqref{e:fromlocdiff} that for some sufficiently small $r_*>0$ the distance between the images under $\gamma_0$ of $\operatorname{supp}\theta_0$ and of the closed ball $\overline{B_{}(x_*,r_*)}$ is strictly positive, namely
    $$d=
    \mathrm{dist}
    \Big(\gamma_0(
    \operatorname{supp}
    \theta_0),\gamma_0\big(\overline{B_{}(x_*,r_*)}\big)\Big)>0.$$
Introduce the following sets
    	\begin{align}
    	    \mathcal{K}_1=\big\{ x \in \mathbb R^2 \; | \; x \text{  is at most } d/4 \text{ away from } \gamma_0(\operatorname{supp}\theta_0)\big\} 
    	\end{align}
and
\begin{align}
    \mathcal{K}_2=\big\{ x \in \mathbb R^2 \; | \; x \text{  is at most } d/4 \text{ away from }\gamma_0(\overline{B_{}(x_*,r_*)})  \big\}.
\end{align}	
     
	By selecting $R_2 \in (0,R_1)$ we can arrange things so that for every $\gamma, \gamma' \in E_e^{\beta}\big(B_{}^s(\theta_0,R_2)\big)$, we have the following estimates
	\begin{align}
	|\gamma(x)-\gamma(y)| \lesssim \|D\gamma\|_{\infty}|x-y|\lesssim L |x-y|,\quad \forall x,y \in \mathbb R^2 \quad \end{align}
  where $L>0$ is a constant and
  \begin{align}
 \|\gamma-\gamma'\|_{L^\infty} \leq \min\{1,d/4\}
	\end{align}
    which follow by combining the mean value estimate, the Sobolev lemma and the triangle inequality.
In particular, this ensures that
	\[
	\gamma(\operatorname{supp}\theta_0) \subseteq \mathcal{K}_1 \quad \mbox{and} \quad \gamma\big(\overline{B(x_*,1)}\big) \subseteq \mathcal{K}_2
	\]
	for any $\gamma \in E_e^{\beta}\big(B_{}^s(\theta_0,R_2)\big)$. 
    
   Next, consider the Taylor expansion with the integral remainder
	\begin{align}
	E_e^{\beta}(\theta + h) = E_e^{\beta}(\theta) + d E_{e}^{_{\beta}} (\theta)h + \int_0^1 (1-t) d^2 E_e^{\beta}{(\theta + t h)}(h,h) \;dt
	\end{align}
    where $h\in H^s(\mathbb R^2)$.
Choose $ R_3 \in(0, R_2)$ and observe that there exists $M > 0$ such that
	\[
	\|d^2 E_e^{\beta}(\theta)(h_1,h_2)\|_{H^s} \leq M \|h_1\|_{H^s} \|h_2\|_{H^s}
	\]
	and
	\[
	\|d^2 E_e^{\beta}(\theta_1)(h_1,h_2)-d^2 E_e^{\beta}(\theta_2)(h_1,h_2)\|_{H^s} \leq M \|\theta_1-\theta_2\|_{H^s} \|h_1\|_{H^s} \|h_2\|_{H^s}
	\]
	for all $\theta, \theta_1, \theta_2 \in B_{}^s(\theta_0,R_3)$ and $h_1,h_2 \in H^s(\mathbb R^2)$. Next, we select $R_{*}$ within the interval  $ (0, R_3)$ such that it satisfies the condition
	\begin{equation}\label{condition}
		\max\{C M R_*^2,C M R_*\} < \kappa_{*}/8.
	\end{equation}
	Now, we choose $R\in(0, R_*)$ and construct two sequences $(\theta_0^{(1,n)})_{n \geq 1}$ and $( \theta_0^{(2,n)})_{n \geq 1}$ as follows.
The initial sequence is defined as
	\[
	\theta_0^{(1,n)}=\theta_0 + \vartheta^{(n)}
	\]
	where $\vartheta^{(n)}$ in $H^s(\mathbb R^2)$ is chosen arbitrarily such that $\|\vartheta^{(n)}\|_{H^s}={R}/{2}$ and its support is contained within $B_{}(x_*,r_n)$, where
	\[
	r_n=\frac{\kappa_*}{8Ln}\|w_*\|_{H^s}\,.
	\]
	Therefore, while the total mass of $\vartheta^{(n)}$ remains unchanged, its support gradually contracts. The second sequence is obtained by perturbing the first one with $w_*^{(n)}:=w_*/n$ which introduces a shift in the supports. This results in
	\begin{align}
	 \theta_0^{(2,n)} = \theta_0^{(1,n)} + w_*^{(n)}=\theta_0 + \vartheta^{(n)}+w_*^{(n)}.
	\end{align}
 Taking $N$ sufficiently large we have
	\[
	\theta_0^{(1,n)},\theta_0^{(2,n)} \in B^s(\theta_0,R)\quad \mbox{ and }\quad r_n \leq 1, \; \forall n \geq N.
	\]
	By construction we have
	\[
	\lim_{n \to \infty} \|\theta_0^{(1,n)}- \theta_0^{(2,n)}\|_{H^s} = \lim_{n \to \infty} \|w_*^{(n)}\|_{H^s} = 0.
	\]
	For $n \geq N$, let $\gamma^{(1,n)} =  E_{e}^{\beta}(\theta_0^{(1,n)})$ and $\gamma^{(2,n)} =  E_{e}^{\beta}( \theta_0^{(2,n)})$, we get
	\begin{align*}
        &\|\theta_0^{(1,n)} \circ (\gamma^{(1,n)})^{-1} - \theta_0^{(2,n)} \circ (\gamma^{(2,n)})^{-1}\|_{H^s} \\
		 &  \quad\quad \qquad\qquad\quad\geq
       \|\theta_0^{(1,n)} \circ (\gamma^{(1,n)})^{-1}-\theta_0^{(1,n)} \circ ( \gamma^{(2,n)})^{-1} \|_{H^s} - \|w_*^{(n)} \circ (\gamma^{(2,n)})^{-1}\|_{H^s}
	\end{align*}
	and by \eqref{bdd} we have $\limsup\limits_{n \to \infty} \|w_*^{(n)} \circ ( \gamma^{(2,n)})^{-1}\|_{H^s}=0$. Moreover, observe that
	\begin{align*}
		&\|\theta_0^{(1,n)} \circ (\gamma^{(1,n)})^{-1}-\theta_0^{(1,n)} \circ (\gamma^{(2,n)})^{-1} \|_{H^s}\\
        &= \|(\theta_0 \circ (\gamma^{(1,n)})^{-1}- \theta_0 \circ (\gamma^{(2,n)})^{-1})+(\vartheta^{(n)} \circ (\gamma^{(1,n)})^{-1}-\vartheta^{(n)} \circ (\gamma^{(2,n)})^{-1})\|_{H^s}\\
        &\ge C \|(\theta_0 \circ (\gamma^{(1,n)})^{-1}- \theta_0 \circ (\gamma^{(2,n)})^{-1})\|_{s}+ C\|(\vartheta^{(n)} \circ (\gamma^{(1,n)})^{-1}-\vartheta^{(n)} \circ (\gamma^{(2,n)})^{-1})\|_{H^s}.
	\end{align*}
    The last inequality follows from the reverse triangle inequality \eqref{ineq2} and the fact that the support of the terms $(\theta_0 \circ (\gamma^{(1,n)})^{-1}$ and $ \theta_0 \circ (\gamma^{(2,n)})^{-1})$ is contained in $\mathcal{K}_1$, which implies that their difference $(\theta_0 \circ (\gamma^{(1,n)})^{-1}- \theta_0 \circ (\gamma^{(2,n)})^{-1})$ is also supported in $\mathcal{K}_1$. Similarly, the functions  $(\vartheta^{(n)} \circ (\gamma^{(1,n)})^{-1}$ and $\vartheta^{(n)} \circ (\gamma^{(2,n)})^{-1})$ are supported in $\mathcal{K}_2$, meaning that their difference $(\vartheta^{(n)} \circ (\gamma^{(1,n)})^{-1}-\vartheta^{(n)} \circ (\gamma^{(2,n)})^{-1})$ is also supported in $\mathcal{K}_2$. Therefore, it suffices to show	\[
	\limsup_{n \to \infty} \|\vartheta^{(n)} \circ (\gamma^{(1,n)})^{-1} - \vartheta^{(n)} \circ (\gamma^{(2,n)})^{-1}\|_{H^s} > 0.
	\]
Again, this can be done by showing that the supports of the functions $\vartheta^{(n)} \circ (\gamma^{(1,n)})^{-1}$ and $\vartheta^{(n)} \circ ( \gamma^{(2,n)})^{-1}$ are disjoint so that we can apply \eqref{ineq2}
\begin{align}
    \|\vartheta^{(n)} \circ (\gamma^{(1,n)})^{-1} - \vartheta^{(n)} \circ ( \gamma^{(2,n)})^{-1}\|_{H^s}&\ge C\Big(\|\vartheta^{(n)} \circ (\gamma^{(1,n)})^{-1}\|_{s} +\| \vartheta^{(n)} \circ (\gamma^{(2,n)})^{-1}\|_{H^s}\Big)\\
    &\ge C\|\vartheta^{(n)} \circ (\gamma^{(1,n)})^{-1}\|_{H^s},
\end{align}
and $\|\vartheta^{(n)} \circ (\gamma^{(1,n)})^{-1}\|_{s}$ has a universal positive lower bound using the first of the inequalities in \eqref{bdd}. 

Now, we estimate the quantity $|\gamma^{(1,n)}(x_*)-\gamma^{(2,n)}(x_*)|$. We have 
	\begin{align*}
		\gamma^{(1,n)} &=E_{e}^{\beta}(\theta_0 + \vartheta^{(n)}) \\
		&= E_{e}^{\beta}(\theta_0) + dE_{e}^{\beta}({\theta_0})(\vartheta^{(n)}) + \int_0^1 (1-t) d^2 E_{e}^{\beta}(\theta_0 + t \vartheta^{(n)})(\vartheta^{(n)},\vartheta^{(n)}) \;dt
	\end{align*}
	and 
	\begin{align}
		 \gamma^{(2,n)} &= E_{e}^{\beta}(\theta_0 + \vartheta^{(n)} + w_*^{(n)})\\
        &= E_{e}^{\beta}(\theta_0) + d E_{e}^{\beta}(\theta_0)(\vartheta^{(n)}+w_*^{(n)})\notag \\
		&~~~+ \int_0^1 (1-t) d^2 E_{e}^{\beta}(\theta_0+t(\vartheta^{(n)}+w_*^{(n)}))(\vartheta^{(n)}+w_*^{(n)},\vartheta^{(n)}+w_*^{(n)})\;dt.\notag
	\end{align}
Hence, we have
	\begin{align*}
	 \gamma^{(1,n)}-\gamma^{(2,n)} &= -dE_{e}^{\beta}(\theta_0 )(w_*^{(n)}) 
     - \underbrace{ 2 \int_0^1 (1-t) d^2E_{e}^{\beta}(\theta_0+t(\vartheta^{(n)}+w_*^{(n)}))(w_*^{(n)},\vartheta^{(n)}) \;dt}_{J_1}\\
     &- \underbrace{ \int_0^1 (1-t) d^2E_{e}^{\beta}(\theta_0 + t(\vartheta^{(n)}+w_*^{(n)}))(w_*^{(n)},w_*^{(n)}) \;dt}_{J_2} \\
    & + \underbrace{\int_0^1 (1-t) \left(d^2E_{e}^{\beta}(\theta_0 + t\vartheta^{(n)})(\vartheta^{(n)},\vartheta^{(n)}) - d^2E_{e}^{\beta}(\theta_0 + t(\vartheta^{(n)}+w_*^{(n)}))(\vartheta^{(n)},\vartheta^{(n)})\right) \; dt}_{J_3} .
	\end{align*}
	Applying the previously derived bounds for the second derivatives, we obtain the following estimates:
	\begin{align*}
	\|J_1\|_{H^s} \leq 2K \|w_*^{(n)}\|_{H^s} \|\vartheta^{(n)}\|_{H^s} = \frac{MR}{n} \|w_*\|_{H^s}, 
    \end{align*}
    and
    \begin{align*}
	\|J_2\|_{H^s} \leq \frac{M}{2n^2} \|w_*^{(n)}\|_{H^s}^{2}  \leq \frac{MR}{n} \|w_*\|_{H^s}
	\end{align*}
	and
    \[
	\|J_3\|_{H^s} \le \int_{0}^1(1-t) M \|w_*^{(n)}\|_{H^s} \|\vartheta^{(n)}\|_{H^s}^2 dt=  \frac{M}{2} \|w_*^{(n)}\|_{H^s} \|\vartheta^{(n)}\|_{H^s}^2\le \frac{MR^2}{4n} \|w_*\|_{H^s} 
	\]
	where the bound for $\|J_2\|_{H^s}$ holds for $n \geq N$ by increasing $N$ if necessary. Thus, by the Sobolev lemma and the choice for $R_*$ in \eqref{condition}, we obtain
	\begin{align*}
	|J_1(x_*)|+|J_2(x_*)|+|J_3(x_*)|& \leq  \frac{C M R}{n} \|w_*\|_{H^s} + \frac{C M R}{n} \|w_*\|_{H^s}+\frac{C M R^2}{4n} \|w_*\|_{H^s} \\
    &\le \frac{\kappa_{*}}{2n} \|w_*\|_{H^s}.
	\end{align*}
	Using this inequality we find
	\begin{align*}
	\big|\gamma^{(1,n)}(x_*)-\gamma^{(2,n)}(x_*)\big| &\ge \big|d E_{e}^{\beta}(\theta_0)(w_*^{(n)})(x_*)\big| - \frac{\kappa_*}{2n}\|w_*\|_{H^s}\\
    &\ge \frac{1}{n} \kappa_{*} \|w_*\|_{H^s} - \frac{\kappa_{*}}{2n}\|w_*\|_{H^s}\\
    &= \frac{\kappa_{*}}{2n} \|w_*\|_{H^s}.
	\end{align*}
	By the Lipschitz property of $\gamma^{(1,n)}, \gamma^{(2,n)}$ we have
	\begin{align*}
	\gamma^{(1,n)}(B_{r_n}(x_*)) \subseteq B_{}(\gamma^{(1,n)}(x_*),\frac{\kappa_{*}}{8n} \|w_*\|_{H^s})\quad \text{ and }\quad \gamma^{(2,n)}(B_{r_n}(x_*)) \subseteq B_{}( \gamma^{(2,n)}(x_*),\frac{\kappa_{*}}{8n} \|w_*\|_{H^s}).
	\end{align*}
	This means $\vartheta^{(n)} \circ (\gamma^{(1,n)})^{-1}$ is supported in $B_{}(\gamma^{(1,n)}(x_*),\frac{\kappa_{*}}{8n} \|w_*\|_{H^s})$ and $\vartheta^{(n)} \circ (\gamma^{(2,n)})^{-1}$ is supported in $B_{}(\gamma^{(2,n)}(x_*),\frac{\kappa_{*}}{8n} \|w_*\|_{H^s})$. Since the distance between the centers of support is larger than $\frac{\kappa{*}}{2n}\|v\|_{H^s}$ and the radii of the supports are $\frac{\kappa_{*}}{8n}\|w_*\|_{H^s}$, we can apply \eqref{ineq2}. Thus, combining with \eqref{bdd}, we obtain
	\begin{align*}
		\|\vartheta^{(n)} \circ (\gamma^{(1,n)})^{-1} - \vartheta^{(n)} \circ (\gamma^{(2,n)})^{-1}\|_{H^s} 
        &\geq C' (\|\vartheta^{(n)} \circ (\gamma^{(1,n)})^{-1}\|_{H^s} + \|\vartheta^{(n)} \circ (\gamma^{(2,n)})^{-1}\|_{H^s}) \\
        &\geq \frac{C'}{C} R/2
	\end{align*}
and therefore we have
	\begin{align*}
	\limsup_{n \to \infty} \|\theta_0^{(1,n)} \circ (\gamma^{(1,n)})^{-1} -\theta_0^{(2,n)} \circ (\gamma^{(2,n)})^{-1}\|_{H^s} \geq \overline C R
	\end{align*}
	with $\overline C$ independent of $R \in (0,R_*)$ whereas $\lim\limits_{n \to \infty} \|\theta_0^{(1,n)}- \theta_0^{(2,n)}\|_{H^s} = 0$. This holds for every $R$ within the range $(0, R_*)$ and the proof is completed.
\end{proof}

\section{Non-continuity of the solution map for $C^\alpha$} 

In \cite{I-JJ} the author considers local existence 
and uniqueness of solutions of \eqref{eq:gSQG} 
in the standard $C^{1,\alpha}$ (large) H\"older spaces 
and poses questions about propagation of regularity 
of solutions in this setting. 
Subsequently, in \cite{CJK25} the authors establish 
local wellposedness of \eqref{eq:gSQG} 
in the (little) H\"older $c^{1,\alpha}$ spaces 
with $\beta < \alpha < 1$. 
In this section we prove that 
the data-to-solution map fails 
to be continuous\footnote{For the corresponding 
result in the case of the 3D Euler equations 
using shear flow initial data, see \cite{MY}.}
in the large H\"older space 
for any $0 < \alpha < 1$. 

Recall the definition of the H\"olderian semi-norm, 
namely 
\begin{equation*} 
\|\theta\|_{\alpha}:=\sup_{x\not=y}\frac{|\theta(x) 
- 
\theta(y)|}{|x-y|^\alpha}. 
\end{equation*} 
The corresponding little H\"older space $c^\alpha$ is defined as 
\begin{equation*} 
c^\alpha := 
\left\{ 
f \in C^\alpha: 
\lim_{|h|\to 0} \frac{|f(x+h)-f(x)|}{|h|^\alpha} = 0 
\quad \text{for all} \quad x 
\right\}. 
\end{equation*} 
%


\noindent
{\it Proof of Theorem \ref{holder case}.}\quad 
Since   $\theta_0\in  C^\alpha\setminus c^\alpha$, 
we see that 
there exist $\varepsilon_0 >0$ 
and a sequence of numbers 
$\ell_n \to 0$ such that 
\[ 
\frac{\lvert \theta_0(x_0+ \ell_n) - \theta_0(x_0) \rvert} {\lvert \ell_n \rvert^{\alpha}} \ge \varepsilon_0. 
\] 
By setting $\ell_n=-h_nT_n$ 
(with $h_n\to 0$ and $T_n\to 0$) we have
\[ 
\frac{\lvert \theta_0(x_0- h_nT_n)\rvert}{\lvert h_nT_n \rvert^{\alpha}} \ge \varepsilon_0 
\] 
since $\theta_0(x_0)=0$. 
From the definition of the H\"olderian norm, we have  
\begin{equation*} 
\begin{split} 
\|\theta_n(T_n)-\theta(T_n)\|_{\alpha} 
&= 
\sup_{x\not=y} 
\frac{|(\theta_{n}(T_n,x)-\theta(T_n,x))-(\theta_{n}(T_n,y)-\theta(T_n,y))|}{|x-y|^\alpha}. 
\\ 
\end{split} 
\end{equation*} 
We now choose 
$x=\eta(T_n,x_0)$ and $y=\eta_n(T_n,x_0)$, 
to obtain the following lower bound 
for the above expression 
\begin{equation*} 
\begin{split} 
&\geq 
\frac{|(D\eta_n\theta_{0}\circ\eta_n^{-1}\circ\eta-D\eta\theta_0\circ\eta^{-1}\circ\eta)-(D\eta_n\theta_{n}\circ\eta^{-1}_n\circ\eta_n-D\eta\theta_0\circ\eta\circ\eta_n)|}{|x-y|^\alpha}. 
\\
\end{split} 
\end{equation*} 
Since $D\eta_n=D\eta$ we further estimate 
\begin{equation*} 
    \begin{split} 
 &\geq \|D\eta^{-1}\|^{-1}  \frac{|(\theta_{0}\circ\eta_n^{-1}\circ\eta-\theta_0\circ\eta^{-1}\circ\eta)-(\theta_{n}\circ\eta^{-1}_n\circ\eta_n-\theta_0\circ\eta\circ\eta_n)|}{|x-y|^\alpha}\\
& =\|D\eta^{-1}\|^{-1} 
\frac{|\theta_0(x_0-h_{n}T_n)-\theta_0(x_0))-(\theta(x_0)-\theta_0\circ\eta^{-1}\circ\eta_n(T_n,x_0))|}{|\eta(T_n,x_0)-\eta_n(T_n,x_0)|^\alpha} 
\\ 
& 
=\|D\eta^{-1}\|^{-1} 
\frac{|\theta_0(x_0-h_nT_n)+\theta_0\circ\eta^{-1}\circ\eta_n(T_n,x_0)|}{|h_nT_n|^\alpha}
\end{split} 
\end{equation*} 
where we used the norm in \eqref{lower bound}. 
Recalling the alignment condition \eqref{alignment}, 
we now find that 
\begin{equation*} 
    \theta_0(x_0-h_nT_n)\cdot\left(\theta_0\circ\eta^{-1}\circ\eta_n(T_n,x_0)\right)>0 
\end{equation*} 
for any sufficiently large $n$. 
Thus, we have 
\begin{equation*} 
\begin{split} 
\geq &\|D\eta^{-1}\|^{-1} \frac{|\theta_0(x_0-h_nT_n)|+|\theta_0\circ\eta^{-1}\circ\eta_n(T_n,x_0)|}{|h_nT_n|^\alpha}\\
\geq &\|D\eta^{-1}\|^{-1} \frac{|\theta_0(x_0-h_nT_n)|}{|h_nT_n|^\alpha}\geq \|D\eta^{-1}\|^{-1} \epsilon_0 
\end{split} 
\end{equation*} 
which is the desired estimate. 
\hfill$\square$ 

\begin{remark}
    In the case of the whole space 
    we can employ approximate solutions as in \cite{BL}. 
    First, we mimic the periodic case by introducing 
    suitable spatial cut-offs. 
    Next, we choose the remainder term so that it shifts 
    the main term just as in the periodic case. 
    Since we can take arbitrary $n$, this procedure 
    is essentially trivial and we therefore omit 
    the details. 
\end{remark} 
%


\appendix
\counterwithin{theorem}{section}

    \section{Analyticity in real Banach spaces}\label{appendix_analyticity}
    The notion of analyticity can be extended to the setting of maps between infinite dimensional spaces. Let $E$ and $F$ be Banach spaces over $\mathbb{C}$ and let $U\subset X$ be an open subset. A map $f: U\rightarrow F$ is analytic if for each $x\in U$, $h\in E$ and $y^*\in F^*$ the function $z\rightarrow y^*\big(f(x+zh)\big)$ is an analytic function of the complex variable $z$ when $|z|$ is sufficiently small. Consequently, $y^*\big(f(x+zh)\big)$ can be represented locally as a power series in the $z$ variable and the standard Cauchy estimates apply. Furthermore, if $f$ is locally bounded then $f$ is complex analytic in $U$ if and only if $f$ is G\^ateaux differentiable at each point in $U$. 

If the G\^ateaux derivative $\partial_h f(x)$ is linear in $h\in E$ and continuous in $x\in U$ then it coincides with the standard Fr\'echet differential $df(x_0)h$ and $f$ is Fr\'echet differentiable at $x_0$. 
    
The fundamental theorem of ordinary differential equations and the implicit function theorem in the analytic setting of Banach spaces are well known. Proofs of both of the propositions below can be found for example in \cite{Zeidler1986} or in the paper of Shnirelman \cite{Shn12}.

\begin{proposition}[Fundamental theorem of ODE]\label{prop: analytic_Shnirelman}
	Let $X, Y$ be complex Banach spaces. Consider the equation
	\begin{equation}\label{eqn: Cauchy_prob}
	\frac{d}{dt}\xi(t)=F(\xi(t),\vartheta),\quad \xi(0)=\xi_0,
	\end{equation}
where $t\rightarrow \xi(t) \in X$, $\vartheta\in Y$ and $F: X\times Y\rightarrow X$ is analytic in a neighborhood of $(\xi_0,\vartheta_0)\in X\times Y$. There exist $r>0$ and $T>0$ such that 
\begin{enumerate}
    \item [(i)] if $\|\vartheta-\vartheta_0\|_Y < r$, then there exists a unique solution $\xi(t,\vartheta)$ of \eqref{eqn: Cauchy_prob} for $|t|<T$; 
    \item [(ii)] the map $(t,\vartheta)\mapsto \xi(t,\vartheta)$ from $(-T,T)\times B_r (\vartheta_0)$ to $X$ is analytic, where  $B_r (\vartheta_0)$ denotes the open ball centered at $\vartheta_0$ of radius $r$.
\end{enumerate}

\end{proposition}
\begin{proof}
    See \cite{Shn12}, Theorem 2.2 or \cite{Zeidler1986}, Chapter 4.
\end{proof}

\begin{proposition}\label{prop:AIFT_Shnirelman}
  Let $E, F, G$ be complex Banach spaces, and let $\Phi: E \times F \to G$ be an analytic mapping defined in a neighborhood of a point $(x_0, y_0) \in E \times F$, with $\Phi(x_0, y_0) = z_0$. Suppose that the partial derivative $\frac{\partial \Phi}{\partial x} (x_0, y_0)$ is an invertible linear operator whose image is all of $G $. Then:  

1. There exists a radius $r > 0$ such that for any $y \in F$ and $z \in G$ satisfying $\|y - y_0\|_F \leq r$ and $\|z - z_0\|_G \leq r$, the equation $\Phi(x, y) = z$ has a unique solution $x(y, z)$ in a neighborhood of $x_0$.  

2. The function $x(y, z)$ is analytic with respect to $(y, z)$.  
\end{proposition}


\subsection{Inequalities for fractional Sobolev functions}\label{appendix_ineq}
In this subsection we will establish inequalities of the form 
\[
\|f+g\| \geq C (\|f\|_s+\|g\|_s)
\]
for functions $f,g$ with disjoint support. For fractional $s$ this causes some difficulties as the norm $\|\cdot\|_s$ is defined in a non-local way. For fixed supports we have

\begin{lemma}\label{lemma_fixed_supp}
	Let $s \in \mathbb R$. There is a constant $C > 0$ such that for all $f,g \in C^\infty_c(\mathbb R)$ with $\operatorname{supp} f \subseteq (-3,-1)$ and $\operatorname{supp} g \subseteq (1,3)$ we have
	\[
	\|f+g\|_s^2 \geq C (\|f\|_s^2 + \|g\|_s^2)
	\]
\end{lemma} 
\begin{proof}
	We take $\gamma,\psi \in C^\infty_c(\mathbb R)$ with $\operatorname{supp}\gamma \subseteq (-3.5,-0.5)$ and $\operatorname{supp}\psi \subseteq (0.5,3.5)$ such that $\left.\gamma\right|_{(-3,-1)} \equiv 1$ and $\left.\psi\right|_{(1,3)} \equiv 1$. We then have
	\[
	\|f\|_s = \|\gamma (f+g)\|_s \leq C_1 \|f+g\|_s
	\]
	and similarly
	\[
	\|g\|_s = \|\psi (f+g)\|_s \leq C_2 \|f+g\|_s
	\]
	giving the desired result.
\end{proof}

In the following we will use the fact that the $H^s$-norm is equivalent to the homogeneous $\dot H^s$-norm if we restrict ourselves to functions with support in a fixed compact $K \subseteq \mathbb R$ (see e.g. \cite{BahouriCheminDanchin2011} p. 39). Recall
\[
\|f\|_{\dot H^s}^2 = \int_\mathbb R |\xi|^{2s} |\hat f(\xi)|^2 d\xi
\]
We often also use $f^\lambda(x):=f(x/\lambda)$ for which we have the following scaling property
\[
\|f^\lambda\|_{\dot H^s}^2 = \lambda^{1-2s} \|f\|_{\dot H^s}^2 
\]
We have

\begin{lemma}\label{lemma_vanishing_support1}
	Let $s \geq 0$. Then there is a constant $C > 0$ with the following property:
	For $x,y$ in $\mathbb R$ with $0 < r:=|x-y|/4 < 1$  we have
	\[
	\|f+g\|_s^2 \geq C  (\|f\|_s^2 + \|g\|_s^2)
	\]
	for all functions $f,g \in C^\infty_c(\mathbb R)$ with $\operatorname{supp}f \subseteq (x-r,x+r)$, $\operatorname{supp}g \subseteq (y-r,y+r)$
\end{lemma}

\begin{proof}
	We use the homogeneous norm. Now scaling with $\lambda=(4r)^{-1}$ gives a situation as in Lemma \ref{lemma_fixed_supp}. We have
	\[
	\|f+g\|_{\dot H^s}^2 = \lambda_n^{2s-1} \|f^{\lambda}+g^{\lambda}\|_{\dot H^s}^2 
	\]
	Now by Lemma \ref{lemma_fixed_supp} we then get
	\[
	\|f+g\|_{\dot H^s}^2 \geq C \lambda^{2s-1} \left(\|f^{\lambda}\|_{\dot H^s}^2+\|g^{\lambda}\|_{\dot H^s}^2 \right)
	\]
	Scaling back gives
	\[
	\|f+g\|_{\dot H^s}^2 \geq C (\|f\|_{\dot H^s}^2+\|g\|_{\dot H^s}^2)
	\]
	This establishes the lemma.
\end{proof}

We will encounter Lemma \ref{lemma_vanishing_support1} also for some negative values of $s$. In these cases we will use

\begin{lemma}\label{lemma_vanishing_support2}
	Let $s <0$ and the same situation as in Lemma \ref{lemma_vanishing_support1}. Then we have 
	\[
	\|f+g\|_s^2 \geq C (\|f\|_s^2+\|g\|_s^2)
	\]
	for all functions $f,g \in C^\infty_c(\mathbb R)$ with $\operatorname{supp}f \subseteq (x-r,x+r)$, $\operatorname{supp}g \subseteq (y-r,y+r)$
\end{lemma}

\begin{proof}
	We claim that for functions with support in some fixed compact set $K \subseteq \mathbb R$ the homogeneous norm
	\[
	\|f\|_{\dot H^s}^2 = \int_\mathbb R |\xi|^{2s} |\hat f(\xi)|^2 d\xi
	\]
	is equivalent to the non-homegeneous norm $\|\cdot\|_s$. One then can argues as in Lemma \ref{lemma_vanishing_support1} by scaling to a fixed situation as in Lemma \ref{lemma_fixed_supp}. So it remains to show the equivalence of the norms. We clearly have $\|\cdot\|_s \leq \|\cdot\|_{\dot H^s}$ since
	\[
	\int_\mathbb R (1+\xi^2)^s |\hat f(\xi)|^2 d\xi \leq \int_\mathbb R \xi^{2s} |\hat f(\xi)|^2 d\xi
	\]
	For the other direction we use the dual definition of the Sobolev norm
	\[
	\|f\|_s = \sup_{\|g\|_{-s} \leq  1} |\langle f,g \rangle|
	\]
	and analogously for the homogeneous norm. Taking $\psi \in C_c^\infty(\mathbb R)$ with $\psi=1$ on $K$ we have for $f$ with support in $K$
	\[
	\|f\|_{\dot H^s} = \sup_{\|g\|_{\dot H^{-s}}\leq 1} |\langle f, \psi \cdot g \rangle|
	\]
	Now note that we have equivalence of the norms $\|\cdot\|_{-s}$ and $\|\cdot\|_{\dot H^{-s}}$ for functions with support in some fixed compact. Therefore
	\begin{eqnarray*}
		\|f\|_{\dot H^s} &=& \sup_{\|g\|_{\dot H^s} \leq 1} |\langle f, \psi \cdot g \rangle| \leq \sup_{g,\|\psi \cdot g\|_{-s} \leq C_1} |\langle f, \psi \cdot g \rangle| \\ &\leq& C_1 \sup_{\|g\|_{-s} \leq 1} |\langle f, g \rangle|
		=\|f\|_{-s}
	\end{eqnarray*}
	showing the equivalence.
\end{proof}


\vspace{0.5cm}
\noindent
{\bf Acknowledgments.}\ 
GM thanks the Hitotsubashi University for providing excellent conditions to finish the work on the paper.
Research of  TY  was partly supported by the JSPS Grants-in-Aid for Scientific Research  24H00186.

\vspace{0.5cm}
\noindent
\textbf{Conflict of Interest.}
The authors declare that they have no conflict of interest.

\bibliographystyle{plain}
\bibliography{NUgSQG}

@article{BL,
  author = {J. Bourgain and D. Li},
  title = {Galilean boost and non-uniform continuity for incompressible {E}uler},
  journal = {Commun. Math. Phys.},
  volume = {372},
  year = {2019},
  pages = {261-380}
}

@article{BL2015,
  author    = {J. Bourgain and D. Li},
  title     = {Strong ill-posedness of the incompressible {E}uler equation in borderline {S}obolev spaces},
  journal   = {Inventiones Mathematicae},
  volume    = {201},
  year      = {2015},
  pages     = {97--157},
  doi       = {10.1007/s00222-015-0601-0}
}

@article{CCCGW12,
  author = {D. Chae and P. Constantin and D. Córdoba and F. Gancedo and J. Wu},
  title = {Generalized surface quasi-geostrophic equations with singular velocities},
  journal = {Comm. Pure Appl. Math.},
  volume = {65},
  year = {2012},
  pages = {1037-1066}
}

@article{CJK25,
  author = {Y-P. Choi and J. Jung and J. Kim},
  title = {On well/ill-posedness for the generalized surface
quasi-geostrophic equation in H\"older spaces},
  journal = {Journal of Differential Equations},
  volume = {443},
  year = {2025},
  pages = {113521}
}

@article{EbMa,
  author = {D. Ebin and J. Marsden},
  title = {Groups of diffeomorphisms and the motion of an incompressible fluid},
  journal = {Ann. Math.},
  volume = {92},
  year = {1970},
  pages = {102-163}
}

@article{I-JJ,
  author = {I-J. Jeong},
  title = {Generalized surface quasi-geostrophic equations: wellposedness and dynamical properties},
  journal = {Lecture notes (nr 59), SNU},
  volume = {},
  year = {2024},
  pages = {}
}

@article{MY,
  author = {G. Misiołek and T. Yoneda},
  title = {Continuity of the solution map of the {E}uler equations in {H}ölder spaces and weak norm inflation in {B}esov spaces},
  journal = {Trans. Amer. Math. Soc.},
  volume = {370},
  year = {2018},
  pages = {4709-4730}
}

@article{Shn12,
  author = {A. Shnirelman},
  title = {On the analyticity of particle trajectories in the ideal incompressible fluid},
  journal = {Global Stoch. Anal.},
  volume = {2},
  year = {2012},
  pages = {149-157}
}

@article{Wash,
  author = {P. Washabaugh},
  title = {The {SQG} equation as a geodesic equation},
  journal = {Arch. Rational Mech. Anal.},
  volume = {222},
  year = {2016},
  pages = {1269-1284}
}

@article{YZJ,
  author = {H. Yu and X. Zheng and Q. Jiu},
  title = {Remarks on well-posedness of the generalized surface quasi-geostrophic equation},
  journal = {Arch. Rational Mech. Anal.},
  volume = {232},
  year = {2019},
  pages = {265-301}
}

@article{constantin2016contrast,
  author    = {P. Constantin and I. Kukavica and V. Vicol},
  title     = {Contrast between {L}agrangian and {E}ulerian analytic regularity properties of {E}uler equations},
  journal   = {Ann. Inst. Henri Poincaré C},
  volume    = {33},
  pages     = {1569--1588},
  year      = {2016}
}

@article{inci2015regularity,
  author = {H. Inci},
  title = {On the regularity of the solution map of the incompressible {E}uler equation},
  journal = {Dyn. Partial Differ. Equ.},
  volume = {12},
  pages = {97--113},
  year = {2015}
}

@article{Inci_K_T2013,
  author    = {H. Inci and T. Kappeler and P. Topalov},
  title     = {On the regularity of the composition of diffeomorphisms},
  journal   = {Memoirs of the American Mathematical Society},
  volume    = {226},
  number    = {1062},
  year      = {2013},
  doi       = {10.1090/memo/1062}
}

@article{jolly2021sqg,
  author = {M. Jolly and A. Kumar and V. Martinez},
  title = {On the existence, uniqueness and smoothing of solutions to the generalized {SQG} equations in critical {S}obolev spaces},
  journal = {Commun. Math. Phys.},
  volume = {387},
  pages = {551--596},
  year = {2021},
  eprint = {arXiv:2101.07228v1}
}

@book{Zeidler1986,
  author    = {E. Zeidler},
  title     = {Nonlinear Functional Analysis and its Applications: I: Fixed-Point Theorems},
  year      = {1986},
  publisher = {Springer},
  edition   = {1986th},
  translator = {P. R. Wadsack},
  isbn      = {978-0387964992}
}

@article{CMT94,
  author    = {P. Constantin and A. Majda and E.  Tabak},
  title     = {Singular front formation in a model for quasigeostrophic flow},
  journal   = {Physics of Fluids},
  volume    = {6},
  number    = {1},
  year      = {1994},
  pages     = {9--11},
  doi       = {10.1063/1.868273}
}

@article{MV2023,
  author    = {G. Misiołek and X.-T. Vu},
  title     = {On Continuity Properties of Solution Maps of the Generalized {SQG} Family},
  journal   = {Vietnam Journal of Mathematics},
  year      = {2023},
  doi       = {10.1007/s10013-023-00633-6}
}

@book{BahouriCheminDanchin2011,
  author = {Bahouri, H. and Chemin, J.-Y. and Danchin, R.},
  title = {Fourier Analysis and Nonlinear Partial Differential Equations},
  publisher = {Springer},
  address = {Berlin, Heidelberg},
  year = {2011},
  series = {Grundlehren der mathematischen Wissenschaften},
  doi = {10.1007/978-3-642-16830-7},
  isbn = {978-3-642-16830-7}
}

@article{Whittlesey1965,
  author    = {E. Whittlesey},
  title     = {Analytic functions in Banach spaces},
  journal   = {Proceedings of the American Mathematical Society},
  volume    = {16},
  pages     = {1077--1083},
  year      = {1965}
}

@article{castro2025unstable,
  title={Unstable vortices, sharp non-uniqueness with forcing, and global smooth solutions for the {SQG} equation},
  author={A. Castro and D. Faraco and F. Mengual and M. Solera},
  journal={arXiv preprint arXiv:2502.10274},
  year={2025}
}

@article{bauer2024geometric,
  title={Geometric analysis of the generalized surface quasi-geostrophic equations},
  author={M. Bauer and P. Heslin and G. Misio{\l}ek and S. Preston},
  journal={Mathematische Annalen},
  volume={390},
  number={3},
  pages={4639--4655},
  year={2024},
  publisher={Springer}
}


\end{document}